\documentclass[reqno]{amsart}

\usepackage{amsmath}
\usepackage{amsthm}
\usepackage{amsfonts}
\usepackage{amssymb}
\usepackage{enumitem}
\usepackage{hyperref}
\usepackage{bbm}
\usepackage[nocompress]{cite}
\usepackage{bm}

\usepackage{todonotes}

\newcommand{\indicator}[1]{\ensuremath{\mathbf{1}_{\{#1\}}}}

\DeclareMathOperator{\tr}{tr}
\DeclareMathOperator{\diag}{diag}

\DeclareMathOperator{\rank}{rank}
\DeclareMathOperator{\supp}{supp}

\newcommand{\Prob}{\mathbb{P}}
\newcommand{\E}{\mathbb{E}}

\renewcommand{\P}{\mathbb{P}}

\newcommand{\x}{\bm{x}}

\def\R{\mathbb{R}}

\theoremstyle{plain}
\newtheorem{theorem}{Theorem}[section]

\newtheorem{lemma}[theorem]{Lemma}
\newtheorem{corollary}[theorem]{Corollary}

\newtheorem{proposition}[theorem]{Proposition}

\theoremstyle{definition}
\newtheorem{definition}[theorem]{Definition}

\theoremstyle{remark}
\newtheorem{remark}[theorem]{Remark}

\begin{document}
	
	\title[Linear Independence of Random Boolean Tensors]{Linear Independence of Random Boolean Tensor Powers at the Dimension Threshold}

	\author[K. Luh]{Kyle Luh}
	\address{Department of Mathematics\\ University of Colorado\\ Campus Box 395\\ Boulder, CO 80309-0395\\USA}
	\email{kyle.luh@colorado.edu}
	\thanks{K. Luh was supported in part by Simons Grant MP-TSM-00001988 and DOE Office of Science Advanced Scientific Computing Research (ASCR) program, funding opportunity DE-FOA-0003432}

	\begin{abstract}
		Let $d \geq 1$ be fixed and let 
		\[
		D(n,d) := \sum_{j=0}^{d} \binom{n-1}{j}.
		\]	
		We show that if $\x^{(1)}, \dots, \x^{(m)}$ are independent uniform points of $\{\pm 1\}^n$ then uniformly for $m \leq D(n,d)$, there exists a constant $C_d > 0$ such that
		\[
		\P((\x^{(1)})^{\otimes d}, \dots, (\x^{(m)})^{\otimes d} \text{ are linearly independent}) = 1 - O_d\left(\frac{\log^{C_d} n}{n^{1/2}} \right).  
		\]
		This achieves the exact dimensional threshold and answers a question asked by Baldi and Vershynin.
		We discuss applications of the result to the semidefinite relaxation of the cut-polytope and to matrix factorization.  
	\end{abstract}

	\maketitle 
	
	\section{Introduction}
	
	We consider the question of how many independent random Boolean tensor powers can be sampled before linear dependence becomes unavoidable.  Dimension gives an immediate upper bound, but it is not obvious whether a random sample attains that bound.  The difficulty is already present for rank-one matrices: although the
	vectors generating the matrices have independent coordinates, the matrix entries satisfy many algebraic relations. We prove that, for every fixed tensor order, random sampling attains the exact dimension of the Boolean tensor span with probability tending to one. 
	
	Let $\x^{(1)},\ldots,\x^{(m)}$ be independent uniform points of $\{\pm 1\}^n$. We regard each symmetric tensor $(\x^{(i)})^{\otimes d}$ as a vector in $(\R^n)^{\otimes d}$ and ask whether these vectors are linearly
	independent over $\R$. The relevant dimension (see Section \ref{sec:dimension}) is
	\begin{equation}\label{eq:intro-dimension}
		D(n,d):=\dim\text{span}\{x^{\otimes d}:x\in \{\pm 1\}^n \}
		=\sum_{j=0}^{d}\binom{n-1}{j}.
	\end{equation}
	This dimension threshold follows from the relations $x_i^2=1$, which reduce homogeneous degree-$d$ monomials to the distinct characters of degrees at most $d$ and of the same parity as $d$. These characters form a basis of the restricted polynomial space.	For fixed $d$, we have $D(n,d)\sim n^d/d!\sim\dim\text{Sym}^d(\R^n)$.
	In particular, $D(n,2)=1+\binom n2$ as can be seen by observing that the diagonal entries of $xx^\top$
	coincide, while its distinct off-diagonal entries are the products $x_ix_j$.

	Motivated by recent applications of random tensors in theoretical computer science \cite{MR4438378, anderson2014more, anari2018smoothed}, Baldi and Vershynin explicitly asked whether the number of independent random Boolean tensor powers can equal the dimension of their span, without any sample-size deficit. They singled out the question of whether
	$1+\binom{n}{2}$ independent random matrices $x^{(i)}(x^{(i)})^\top$ are linearly independent with high probability \cite[Section 10.2.2, p.46]{BV2018v1}.
	Our main theorem answers this question for every fixed degree. Define
	\begin{equation}\label{eq:failure-definition}
		p_{n,d}(m)=\Prob\!\left[
		(\x^{(1)})^{\otimes d},\ldots,(\x^{(m)})^{\otimes d}
		\text{ are linearly dependent}\right].
	\end{equation}
	
	\begin{theorem}[Independence at the dimension threshold]\label{thm:main}
		There is an absolute constant $B$ such that the following holds. For every
		fixed $d\geq1$, there is a constant $C_d >0$ such that, for
		$1\leq m\leq D(n,d)$,
		\begin{equation}\label{eq:main-bound}
			p_{n,d}(m)\leq C_d\frac{(\log n)^{a_d+1}}{\sqrt n},
			\qquad a_d=Bd\log d.
		\end{equation}
	\end{theorem}
	
	Thus $D(n,d)$ random tensor powers form a basis of their full linear span
	with high probability, whereas $D(n,d)+1$ tensor powers are always
	dependent. The sample-size bound is exact.  The special case $d=1$ of the linear independence of the $D(n,1) = n$ random vectors is equivalent to the invertibility of the $n \times n$ random matrix.  This is a fundamental question that has received much attention since Koml\'{o}s posed the problem in the late 60's \cite{MR238371, MR1260107, MR2480613, MR2557947, MR4076632, MR4356701}.  
	Thus, our result can be seen as a high-dimensional generalization of the classical non-singularity statement for random
	sign matrices.  It is surprising that even for the natural question of $d=2$, which is for random rank-1 matrices, the 
	problem was still open \cite{MR4016132}.
	Our contribution is the exact-dimensional conclusion for all fixed degrees. 
	
	There is a fascinating connection between this random tensor question and coding theory.  Abbe, Shpilka and Wigderson were led to this question of linear independence of rank-1 tensors while studying Reed-Muller codes under random erasures and errors \cite{MR3400278}. 
	\begin{theorem} \label{thm:abbe}
		Let $n,d,m, t$ be positive integers such that
		\[
		m < \sum_{j=0}^d \binom{n - \log(\sum_{i=0}^d \binom{n}{i}) - t }{j}
		\]
		Let $\bm{x}^{(1)}, \dots, \bm{x}^{(m)}$ be independent random vectors uniformly distributed in 
		$\{\pm 1\}^n$.  Then, with probability greater than $1 - 2^{-t}$, the tensors
		\[
		(\x^{(1)})^{\otimes d}, \dots, (\x^{(m)})^{\otimes d}
		\]
		are linearly independent.
	\end{theorem}
	
	\begin{remark}
		For fixed $d$, the above result correctly captures the leading term of the dimension threshold, but is roughly $n^{d-1} \log n$ short of the exact threshold.  For example, specializing to $d =1$, the result says that an $n \times m$ random matrix has full rank with high probability for $m < n - \Omega(\log n)$. This is sub-optimal in the dimension and really does not capture the difficulty of singularity problem.   On the other hand, Theorem \ref{thm:abbe} applies even when $d$ varies with $n$, which is outside the scope of our main result.
	\end{remark}
	
	In the next few subsections, we sketch the key ideas of the proof of Theorem \ref{thm:main}.
	
	\subsection{Polynomial evaluations and a near-square rank estimate}
	The first, important reduction is to convert the tensor problem into the language of Boolean analysis. 
	For $S \subset [n-1]$, we write
	\[
	\chi_S(y)=\prod_{j\in S}y_j
	\]
	and we define
	\[
	\mathcal{P}_{d}=\text{span}\{\chi_S:|S|\leq d\},\qquad
	\Psi_d(y)=(\chi_S(y))_{|S|\leq d}.
	\]
	The change of variables $y_i=x_ix_n$ gives $x=x_n(y,1)$, which can be seen as a dehomogenization procedure.
	After a nonzero scalar rescaling and deletion of repeated coordinates, $x^{\otimes d}$ is therefore represented by $\Psi_d(y)$. The transformed samples remain independent and uniform. This identifies the tensor problem with full row rank of the evaluation matrix
	\[
	M_m=\bigl(\chi_S(\bm{y}^{(i)})\bigr)_{1\leq i\leq m,\ |S|\leq d}
	\in\R^{m\times D(n,d)}.
	\]
	At $m=D(n,d)$, invertibility means that $D(n,d)$ random evaluations uniquely
	determine every function in $\mathcal{P}_{n-1,d}$, the space of polynomials with degree less than or equal to $d$. This conversion to polynomials is motivated by the polynomial-threshold-function approach of
	Baldi and Vershynin \cite{MR4016132}.
	
	The first crucial result is an estimate on the rank when $M_m$ is nearly square. For $z\in\mathbb{F}_2^{n-1}$, define
	\[
	\Phi_d(z)=(z^S)_{|S|\leq d},\qquad z^S=\prod_{j\in S}z_j,
	\]
	All products here are taken in $\mathbb{F}_2$. 
	
	\begin{proposition} \label{prop:firstmoment}
		Fix $d\geq1$. There are $C_d,c_d>0$ such that, for all sufficiently large
		$n$, all $1\leq m\leq D(n,d)$, and independent uniform $\bm{z}^{(1)},\ldots,\bm{z}^{(m)}\in\mathbb{F}_2^{n-1}$,
		\begin{equation}\label{eq:intronearsquare}
			\Prob\left(\rank_{\mathbb{F}_2}
			[\Phi_d(\bm{z}^{(1)})\ \cdots\ \Phi_d(\bm{z}^{(m)})]<m\right)
			\leq 2^{-(D(n,d)-m)}+C_dm2^{-(n-1)/2}+e^{-c_d(n-1)^d}.
		\end{equation}
		The same bound holds for failure of full row rank of $M_m$ over $\R$.
	\end{proposition}
	
	For fixed degree, any additive deficit $D(n,d)-m$ tending to infinity is	therefore sufficient for high-probability independence. In particular,
	$m=D(n,d)-\lceil A\log_2 n\rceil$ gives failure probability $O_{d,A}(n^{-A})$ for every fixed $A>0$. The statement about rank over $\R$ follows from an invertible change of monomial basis and the fact that a nonzero minor modulo $2$ is nonzero over $\R$.
	
	A precursor to this result is in \cite{MR3400278}. Their column-independence theorem $\Theta_{d,A}(n^{d-1}\log n)$ short of the dimension threshold.
	Thus, for $d\geq2$, Proposition \ref{prop:firstmoment} improves the additive deficit substantially from order $n^{d-1}\log n$ to order $\log n$ at polynomially small failure probability.

	\subsection{From a global dependence count to the last sample}
	The mechanism behind \eqref{eq:intronearsquare} is a subtle first-moment
	calculation. Let $\bm{N}_{\mathrm{dep}}$ count the non-zero binary relations among the sampled columns, and let $\mathcal{Q}_{n-1,d}$ be the space of multilinear $\mathbb{F}_2$-polynomials of degree at most $d$. For
	$\beta(q)=\E_{\bm{z}\sim\mathbb{F}_2^{n-1}}[(-1)^{q(\bm{z})}]$, character orthogonality gives
	\begin{equation}\label{eq:introfirstmomentidentity}
		\E\bm{N}_{\mathrm{dep}}
		=2^{-K}\sum_{q\in\mathcal{Q}_{r,d}}\bigl((1+\beta(q))^m-1\bigr).
	\end{equation}
	This sum can be controlled by dividing into three cases.
	The zero polynomial contributes at most $2^{m-D(n,d)}$, small positive biases are controlled by their second moment, and large positive biases are counted using the weight-distribution estimate of Abbe, Shpilka and Wigderson~\cite{MR3400278}.

	To complete the argument, we require information not supplied by the dimension of the kernel alone. Once $m_0=D(n,d)-s$ samples have been exposed, with
	$s \approx \log n$, we let
	\[
	\mathcal{W}_0=\{f\in\mathcal{P}_{n-1,d}:f(\bm{y}^{(i)})=0\text{ for }i\leq m_0\}.
	\]
	The near-square estimate gives $\dim\mathcal{W}_0=s$ with high probability.
	We also prove, outside an event of probability $e^{-c'_d n^d}$, that \emph{every} nonzero $f\in\mathcal{W}_0$ has at least $\kappa_d n$ nonzero coefficients in its degree-$d$ part $f^{=d}$ supported on pairwise
	disjoint sets of variables. This simultaneous statement is essential.
	The polynomials selected during the later exposures depend on the previous samples.
	
	To establish this structure, we observe that a small matching of highest-degree monomials would place the polynomial in a subspace whose degree-$d$ terms all meet one small coordinate set. Each such deterministic subspace has small dimension compared with the number of initial samples. 
	A non-zero degree-at-most-$d$ polynomial is non-zero on at least a $2^{-d}$ fraction of the cube, so enough
	independent evaluations annihilate the common kernel on each subspace with exponentially small failure probability. A union bound over the coordinate sets excludes all these subspaces at once. This converts a
	statement about fixed spaces into a uniform structural statement about the random kernel.
	
	Finally, the polynomial anti-concentration theorem of Meka, Nguyen, and
	Vu \cite{MR3542863} then bounds the probability that a nonzero polynomial with this structure vanishes at a fresh, random point by $C_d(\log n)^{a_d}/\sqrt n$. Every later kernel is contained in $\mathcal{W}_0$, so the structural conclusion still holds. Each newly exposed sample therefore reduces the kernel dimension. A union bound over the final $s$ samples gives \eqref{eq:main-bound}. The additional logarithm comes from this union
	bound. The two parts of the argument address different obstructions.  The first controls excess kernel dimension, while the second prevents the remaining real kernel from surviving the last evaluations.

	\subsection{Geometric consequences}
	We mention two applications of the special case of $d=2$. The first concerns simplicial faces of the set of correlation matrices, whose structure was studied by Laurent and Poljak \cite{laurent1996facial} and whose random construction was developed by Tropp \cite{tropp2018simplicial}. Our quadratic theorem shows that random sign vertices produce faces of the largest possible simplicial dimension with high probability. It reaches the exact geometric threshold, removing the oversampling in earlier guarantees. 
	
	For correlation matrices, these simplicial faces identify portions of the cut polytope preserved exactly by its semidefinite relaxation and underlie existing sign-component recovery algorithms. Our quadratic theorem shows that random sign vertices generate such faces at the exact geometric dimension threshold, yielding maximum-dimensional simplicial faces with high probability.  The cut polytope is a fundamental object in combinatorial optimization, but is computationally difficult to work with directly so understanding properties of its semidefinite relaxation is of great importance in theory and practice.
	
	The second application is to matrix factorization, which is a fundamental method in statistics and data analysis.
	We extend the range in which Kueng--Tropp's sign-component recovery theorem applies \cite{kueng2021binary}, which is a foundational result in the theory of existence and uniqueness of factorizations.

	\subsection{Conjectured rate and further work}
	
	Using the conjectured polynomial anti-concentration \cite{MR3542863}, for which the quadratic case was recently resolved \cite{kwan2025resolution}, the probability bound in Theorem \ref{thm:main} can be reduced to $O_d(\log n/ \sqrt{n})$.  However, motivated by the heuristic in random matrix theory, for which the conjecture is that the probability of matrix singularity for Rademacher random matrices is dominated by the probability of having two linearly dependent rows or columns, we conjecture that 
	\[
	p_{n,d} (D(n,d)) = 2^{-n + o(n)}.
	\]
	Of course, the lower bound in this case is that two rank-1 tensors are equal up to sign.  
	
	An important follow-up question is to quantify the linear independence of the boolean tensors. In particular, are
	random tensors \emph{well conditioned}?  This would extend the work of Vershynin \cite{MR4140540} to the dimension threshold.
	In many applications, being well-conditioned translates to robustness, stability and running-time bounds of algorithms \cite{MR4438378, anderson2014more, anari2018smoothed}. 
	
	It would also be of interest to see if the dimension threshold can be reached when $d$ can grow with $n$.     
	
	\subsection{Organization of the remainder}
	In Sections \ref{sec:dimension} we explain the relevant tensor dimensions and the origin of $D(n,d)$.   A brief summary of the notation and some basic results of Boolean analysis are given in Sections \ref{sec:fourieroverR} and \ref{sec:fourieroverF}.  Sections \ref{sec:reduction}, \ref{sec:nearsquare}, \ref{sec:lowdim}, \ref{sec:kernel} and \ref{sec:lastexposure} contain the proof of the main theorem divided into key steps. Finally, Section \ref{sec:applications} contains the details of the geometric consequences of Theorem \ref{thm:main}.

	\section{Relevant Dimensions} \label{sec:dimension}
	We now explain the origin of $D(n,d)$. 
	The ambient symmetric tensors space $\text{Sym}^d(\R^n)$ has dimension 
	\[
	N(n,d) = \binom{n+d-1}{d}
	\]
	as can easily be seen by identifying symmetric order-$d$ tensors with 
	homogeneous polynomials.  However, on the Boolean cube, one has the relation
	$x_i^2 = 1$ so for a multi-index $\alpha = (\alpha_1, \dots, \alpha_n)$ with 
	$\sum_i \alpha_i = d$, 
	\[
	x^\alpha = \prod_{i: \alpha_i \text{ odd}} x_i = \chi_{S(\alpha)}(x).
	\]
	where
	\[
	S(\alpha) := \{i: \alpha_i \text{ is odd} \}
	\]
	and
	\[
	\chi_S(x) = \prod_{i \in S} x_i.
	\]
	Since $\alpha_1 + \dots + \alpha_n = d$ the number of 
	odd exponents has the same parity as $d$.  Conversely, every subset 
	$S \subset [n]$ satisfying the above conditions occurs.  One way this can be 
	accomplished is by setting the exponent to 1 for every coordinate in $S$
	and then choosing the exponent $d - |S|$, which is even, for some coordinate outside of $S$.   Thus, this restricted set of monomials spans the relevant space.  Its dimension is
	\begin{align*}
		D(n,d) &:= \sum_{\substack{0 \leq k \leq d \\ k \equiv_2 d}} \binom{n}{k} \\
		&= \sum_{\substack{0 < k \leq d\\ k \equiv_2 d}} \binom{n-1}{k}
		+ \sum_{\substack{0 < k \leq d\\ k \equiv_2 d}} \binom{n-1}{k-1} \\
		&= \sum_{\substack{0 < j \leq d\\ j \equiv_2 d}} \binom{n-1}{j}
		+ \sum_{\substack{0 < j \leq d\\ j+1 \equiv_2 d}} \binom{n-1}{j} \\
		&= \sum_{j=0}^{d} \binom{n-1}{j} \\
	\end{align*}
	
	As a more concrete illustration of this dimension, when $d = 2$, $D(n,2) = \sum_{j=0}^2 \binom{n-1}{j} = 1 + \binom{n}{2}$.  This can 
	easily be seen since for $x \in \{\pm 1\}^n$, the matrix $x^{\otimes 2} = x x^\top$ has diagonal 
	entries equal to one so the only non-trivial coordinates are the $\binom{n}{2}$ pairwise products
	$x_i x_j$ for $i < j$ and the constant diagonal coordinate.  
	
	\section{Basic Fourier Analysis of Boolean Functions over $\R$} \label{sec:fourieroverR}
	We follow the conventions of \cite[Sections 1.2-1.5]{o2014analysis}.  For $S \subset [r]$, we let
	\[
	\chi_{S}(y) = y^S = \prod_{i \in S} y_i, \quad y \in \{\pm 1\}^r
	\]
	and by convention $\chi_\emptyset = 1$.  The inner product is normalized by the uniform measure so that
	\[
	\langle f, g \rangle = \E_{y \sim \{\pm 1\}^r} [f(y) g(y)]  = 2^{-r} \sum_{y \in \{\pm 1\}^r} f(y) g(y).  
	\]
	The characters $\chi$ form an orthonormal basis.  Thus, 
	\[
	f(y) = \sum_{S \subset [r]} \hat{f}(S) \chi_S(y), \quad \text{where} \quad \hat{f}(S) = \langle f, \chi_S \rangle.  
	\]
	Note that in particular,
	\[
	\hat{f}(\emptyset) = \E[f], \quad \|f\|_2^2 = \sum_{S \subset [r]} \hat{f}(S)^2.
	\]
	We let
	\[
	f^{=k} = \sum_{|S| = k} \hat{f}(S) \chi_S \quad \text{and} \quad f^{\leq d} = \sum_{|S| \leq d} \hat{f}(S) \chi_S.
	\]
	For non-zero $f$, its Fourier degree is $\deg(f)  = \max\{|S|: \hat{f}(S) \neq 0\}$.  
	We also let
	\[
	\mathcal{P}_{r,d} = \{f: \{\pm 1\}^r \rightarrow \R: f = f^{\leq d}\}
	= \text{span}_\R \{\chi_S: |S| \leq d\}
	\]
	and set 
	\[
	K(r,d) = \dim \mathcal{P}_{r,d} =  \sum_{j=0}^{r} \binom{r}{j}.
	\]

	\section{Boolean analysis over $\mathbb{F}_2$} \label{sec:fourieroverF}
	In this section we fix the notation needed for the estimates from coding theory \cite{MR3400278}.
	We distinguish the Fourier degree from the previous section from degree over $\mathbb{F}_2$.  Let
	\[
	\mathcal{Q}_{r,d} = \{q: \mathbb{F}_2^r \rightarrow \mathbb{F}_2: \deg_{\mathbb{F}_2}(q) \leq d\}.
	\]
	Each $q \in \mathcal{Q}_{r,d}$ has a unique multilinear representation
	\[
	q(z) = \sum_{\substack{S \subset [r] \\ |S| \leq d}} c_S(q) z^S
	\]
	where $c_S(q) \in \mathbb{F}_2$ and $z^S =  \prod_{i \in S} z_i$.
	
	As in \cite{o2014analysis}, we use the sign encoding
	\[
	\chi(b) = (-1)^b \text{ for } b \in \mathbb{F}_2
	\]
	and
	\[
	\chi(z) = (\chi(z_1), \dots, \chi(z_r)) \in \{\pm 1\}^r.
	\]
	Thus, $y = \chi(z)$ is equivalent to $y_i = 1 - 2z_i$ when $z_i$ is represented by $0$ or $1$ in $\R$.
	
	On $\mathbb{F}_2^K$, the parity character is 
	\[
	\chi_S(z) = (-1)^{\sum_{i \in S} z_i} =\chi_S(\chi(z)).
	\] 
	For $v \in \mathbb{F}_2^K$ and a dual index $\gamma \in \widehat{\mathbb{F}_2^K}$, we likewise write $\chi_\gamma(v) = (-1)^{\gamma \cdot v}$ where the dot product is over $\mathbb{F}_2$.  These 
	sign-valued characters must not be confused with the $\mathbb{F}_2$-valued monomials $z^S$.

	\section{Reduction to rank of Fourier evaluation matrix} \label{sec:reduction}
	We first convert this tensor problem into a question about the 
	rank of a random Fourier-evaluation matrix.  
	For deterministic points $y^{(1)}, \dots, y^{(m)} \in \{\pm 1\}^r$ define
	\[
	\mathsf{Ev}_m: \mathcal{P}_{r,d} \rightarrow \R^m \quad \text{ by } \quad \mathsf{Ev}_m f = (f(y^{(1)}), \dots, f(y^{(m)})).
	\]
	We let 
	\[
	\Psi_d(y_i) := (\chi_S(y_i))_{|S| \leq d} \in \R^{K(r,d)}.
	\]
	In the Fourier basis, the matrix of $\mathsf{Ev}_m$ is
	\[
	M_m = (\chi_S (y^{(i)}))_{i,S} =
	\left[ \begin{array}{c}  \Psi_d(y^{(1)})^{\mathsf{T}} \\
		\vdots \\
		\Psi_d(y^{(m)})^{\mathsf{T}} \end{array}\right] \in \R^{m \times K(r,d)}.
	\]

	We now reduce linear independence of the tensor powers to full row rank of Fourier-evaluation matrix, $M_m$. This viewpoint is motivated by the polynomial lift used by Baldi and Vershynin \cite[Sections 3--4]{MR4016132}
	to study polynomial threshold functions and random tensors, and by the monomial evaluation matrices underlying Reed--Muller codes \cite{MR3400278}.

	\begin{lemma} \label{lem:reduction}
		Let $x^{(1)}, \dots, x^{(m)} \in \{\pm 1\}^n$.
		Set
		\[
		y^{(k)} = (x_{1}^{(k)} x_{n}^{(k)}, \dots, x_{n-1}^{(k)} x_{n}^{(k)}) \in \{\pm 1\}^{n-1}.
		\]
		Then 
		\[
		\dim \text{span}(\x_1^{\otimes d}, \dots, \x_m^{\otimes d}) = \rank[\Psi_d(y^{(1)}) \dots \Psi_d(y^{(m)})].
		\]
	\end{lemma}
	\begin{proof}
		Let $x \in \{\pm 1\}^n$ be fixed and set $y_i = x_i x_n$ for $ i < n$.  A tensor coordinate
		is indexed by a $d$-tuple $(i_1, \dots, i_d)$.  If $\alpha_j$ denotes the number of occurrences
		of $j < n$ in this tuple, then
		\[
		x_{i_1} \dots x_{i_d} = x_n^d \prod_{j=1}^{n-1} y_j^{\alpha_j} = x_n^d \chi_{S(\alpha)}(y)
		\]
		where $S(\alpha) = \{j< n: \alpha_j \text{ is odd}\}$. 
		Thus, every coordinate of $x_n^{-d} x^{\otimes d}$ is a coordinate of $\Psi_d(y)$.  
		
		Conversely, every $\chi_S(y)$ with $|S| \leq d$ occurs: take one factor $x_j$ for each $j \in S$ and let the remaining $d - |S|$ positions be filled with $x_n$.  Thus, there is a fixed linear map 
		\[
		A: \R^{K(n-1,d)} \rightarrow \R^{n^d}
		\]
		such that
		\[
		x_n^{-d} x^{\otimes d} = A \Psi_d(y).  
		\]
		A coordinate projection $B: \R^{n^d} \rightarrow \R^{K(n-1,d)}$ 
		that selects one representative coordinate for each $S$ satisfies $BA = I$, so $A$ is injective.
		Applying $A$ to the columns preserves their rank as does multiplying the $k$-th column by the non-zero scalar $(x_n^{(k)})^{-d}$.  
		This proves the statement.  
	\end{proof}
	
	Note that as $y$ ranges over $\{\pm 1\}^{n-1}$, the vectors $\Psi_d(y)$ span $\R^{K(n-1,d)}$ since the 
	characters are independent as functions. 
	If $\x$ is uniform on the hypercube $\{\pm 1\}^n$ then
	\[
	(\x_1 \x_n, \dots, \x_{n-1} \x_n)
	\]
	is uniform on $\{\pm 1\}^{n-1}$.  We have therefore reduced the random tensor problem to 
	a random character-evaluation problem.
	
	\begin{theorem} \label{thm:characterevaluation}
		Fix $d \geq 1$, let $n \rightarrow \infty$, set $K = K(n-1,d)$ and define $a_d = B d \log d$, where $B$ is from Theorem \ref{thm:main}.  If
		$Y_1, \dots, Y_K$ are independent uniform points on $\{\pm 1\}^{n-1}$ then
		\[
		\P(M_{K} \text{ is singular}) \leq C_d \frac{(\log n)^{a_d + 1}}{\sqrt{n}}.  
		\]
		The same bound applies for any $m \leq K$.  
	\end{theorem}
	
	In the remainder of this article, we prove Theorem \ref{thm:characterevaluation}.  
	
	\section{Near-Square Estimate} \label{sec:nearsquare}
	Let $\bm{z}$ be uniform on $\mathbb{F}_2^{n-1}$ and define 
	\[
	\Phi_d(\bm{z}) := (\bm{z}^S)_{S \subset [n-1], |S| \leq d} \in \mathbb{F}_2^{K(n-1,d)}
	\]
	where $\bm{z}^S := \prod_{i \in S} \bm{z}_i$. For $q \in \mathcal{Q}_{r,d}$ define its 
	\emph{bias} and \emph{weight} by 
	\[ 
	\beta(q) = \E_{\bm{z}} \chi(q(\bm{z}))
	\]
	and
	\[
	\text{wt}(q) = \P(q(\bm{z}) = 1)
	\]
	Thus, 
	\[
	\beta(q) = 1 - 2 \text{wt}(q).  
	\]
	
	\subsection{Low-degree Boolean polynomials}
	\begin{lemma} \label{lem:support}
		For every non-zero $q \in \mathcal{Q}_{n-1,d}$, 
		\[
		\P(q(\bm{z}) = 1) \geq 2^{-d}.  
		\]
		As a consequence, when $d \geq 2$, $q \neq 0$ and $\beta(q) > 0$ then
		\[
		0 < \beta(q) \leq \beta_d := 1 - 2^{1-d} < 1.
		\]
		For $d=1$, no nonzero polynomial has positive bias.  
	\end{lemma}
	
	\begin{proof}
		We prove this by induction on the degree.  The base case is clear for a nonzero constant.  
		Now choose a variable $z_i$ occurring in the highest-degree monomial and write
		\[
		q(z) = z_i a(z_{\hat{i}}) + b(z_{\hat{i}}),
		\]
		where $a \neq 0$ and $\deg a \leq d-1$.  If we choose an assignment so that $a = 1$, the two values of $q$ obtained by setting $z_i = 0$ and $z_i = 1$ are distinct, so one of them is equal to one.  By induction, $a=1$ on at least a $2^{-(d-1)}$ fraction of the lower-dimensional hypercube.  Thus, $q=1$ on at least a $2^{-d}$ fraction of the full cube.  
		
		We also have
		$\beta(q) = 1 - 2 \text{wt}(q) \leq 1 - 2^{1-d}$.  When $d=1$, every non-constant affine 
		polynomial is balanced, while the only remaining non-zero polynomial is the constant $1$ 
		whose bias is $-1$.  		
	\end{proof}
	
	We now recall a useful theorem from \cite{MR3400278}.  We specialize this theorem to our setting 
	and our notation.  In words, the lemma tells us that among all low-degree polynomials over $\mathbb{F}_2$, few have a substantial preference of outputting $0$ over $1$. 
	
	\begin{lemma}\cite[Theorem 3.3]{MR3400278} \label{lem:count}
		Fix $d \geq 1$.  There is a constant $A_d$ such that for every $0 < \tau \leq 1/2$,
		\[
		|\{0 \neq q \in \mathcal{Q}_{n-1,d} : \beta(q) \geq \tau\}| \leq \tau^{-A_d (n-1)^{d-1}}.  
		\]
	\end{lemma}
	\begin{proof}
		For $d =1$, the set is empty by the previous lemma. Now, assume $d \geq 2$.  The condition that
		$\beta(q) \geq \tau$ is equivalent to 
		\[
		\text{wt}(q) \leq \frac{1-\tau}{2}. 
		\] 
		For fixed $d$ and all sufficiently large $n$, Theorem 3.3 of \cite{MR3400278} applied with its parameter $\ell = 1$ and $\varepsilon = \tau$ gives
		\[
		|\{q: \text{wt}(q) \leq (1 - \tau) 2^{-1} \}| \leq (1/\tau)^{8c \sum_{j=0}^{d-1} \binom{n-2}{j}}
		\leq \tau^{-A_d (n-1)^{d-1}}.  
		\]
		Removing the zero polynomial only decreases the count.
	\end{proof}
	
	\subsection{First Moment Count}
	The following proposition tells us that the Fourier-evaluation matrix is full rank even for $m$ quite close to the dimension of the polynomial space.
	
	This initial estimate strengthens the column-independence theorem of Abbe, Shpilka, and Wigderson in the fixed-degree regime \cite[Theorem 4.5]{MR3400278}. As with their tensor theorem, their largest dimension is about $n^{d-1}\log n$ short of the optimal dimension.  Because of this, their estimate is too weak for us to union bound over the remaining $n^{d-1} \log n$ missing columns (see Section \ref{sec:lastexposure}.  Their result follows from sequential exposure and a uniform bound on common zero sets, obtained through generalized Hamming weights. We instead count all binary dependence relations by a first-moment argument and express their expected number in terms of polynomial biases. Applying their weight-distribution estimate to this expression yields a bound governed directly by the dimension deficit $K-m$.  This will allow us to use an initial exposure that stops only $O(\log n)$ columns below the dimension threshold. The subsequent anti-concentration argument then handles the remaining columns.

	\begin{proof}[Proof of Proposition \ref{prop:firstmoment}]
		Let $\bm{v} = \Phi_d(\bm{z})$ and let $\bm{v}^{(1)}, \dots, \bm{v}^{(m)}$ be independent copies
		of $\bm{v}$.  Let $N_{dep}$ count the non-empty subsets $I \subset [m]$ such that
		\[
		\sum_{i \in I} \bm{v}_i = 0 \text{ in } \mathbb{F}_2^{K}.
		\]
		If the columns are dependent, then $N_{dep} \geq 1$.  By Markov's inequality,
		\[
		\P(\rank_{\mathbb{F}_2} [\bm{v}^{(1)} \ldots \bm{v}^{(m)}] < m) \leq \E N_{dep}.
		\]
		For $t \geq 1$, let
		\[
		p_t := \P(\bm{v}^{(1)} + \dots + \bm{v}^{(t)} = 0).  
		\]
		We identify $q \in \mathcal{Q}_{r,d}$ with its coefficient vector 
		\[
		\gamma(q) = (c_S(q))_{|S| \leq d} \in \widehat{\mathbb{F}_2^K}.
		\]
		To be explicit about the normalization, let $\phi(\bm{v}) = 2^K \P(\bm{v} = v)$
		be the probability mass function of $\bm{v}$ relative to the uniform measure on $\mathbb{F}_2^K$.  The Fourier transform gives
		\[
		\widehat{\phi}(\gamma (q))= 2^{-K} \sum_{ v \in \mathbb{F}_2^K} \phi(v) \chi_{\gamma(q)}(v) = \beta(q).
		\]
		Orthogonality of characters and independence now give
		
		\[
		p_t = 2^{-K} \sum_{q \in \mathcal{Q}_{r,d}} \beta(q)^t.
		\]
		Summing over the possible support sizes of a dependency relation,
		\begin{align} \label{eq:firstmoment}
			\E N_{dep} &= \sum_{t=1}^m \binom{m}{t} P_t  \nonumber \\
			&= 2^{-K} \sum_{q \in \mathcal{Q}_{r,d}} ((1 + \beta(q))^m -1).
		\end{align}
		The zero polynomial has bias 1 and contributes at most 
		\[
		2^{-K} (2^m - 1) \leq 2^{m-K} = 2^{-s}.  
		\]
		If $q \neq 0$ and $\beta(q) \leq 0$ then $(1 + \beta(q))^m - 1 \leq 0$ so we can discard
		these terms for the upper bound. We control the sum in \eqref{eq:firstmoment} by dividing into 
		two cases.
		
		\noindent 
		\emph{Case I: Small, positive bias.}
		For $0 \leq u \leq 1/m$,
		\[
		(1+u)^m - 1 \leq emu.
		\]
		Therefore, the contribution of $0 < \beta(q) \leq 1/m$ is at most
		\[
		em 2^{-K} \sum_q \beta(q).  
		\]
		By the Cauchy-Schwarz inequality,
		\[
		2^{-K} \sum_q \beta(q) \leq \left(2^{-K} \sum_q \beta(q)^2 \right)^{1/2}.  
		\]
		If $\bm{z}'$ is an independent copy of $\bm{z}$, the orthogonality of the characters gives
		\[
		2^{-K} \sum_q \beta(q)^2 = \P(\Phi_d(\bm{z}) = \Phi_d(\bm{z}')) = \P(\bm{z}= \bm{z}') = 2^{-(n-1)}.  
		\]
		The second equality used the degree-one
		coordinates of $\Phi_d$.  Thus, the contribution for the polynomials with small, but positive bias 
		is at most 
		\[
		em 2^{-(n-1)/2}.
		\]
		
		\noindent 
		\emph{Case II: Large, positive bias.}
		It remains to address non-zero $q$ with $\beta(q) > 1/m$.  For $d=1$, there are none, so assume
		that $d \geq 2$. and recall that
		\[
		\beta(q) \leq \beta_d = 1 - 2^{1-d}
		\]
		by Lemma \ref{lem:support}.  We divide into dyadic intervals.
		Let
		\[
		\tau_j = \frac{2^j}{m}
		\]
		for those $j$ with $\tau_j \leq 1/2$.  By Lemma \ref{lem:count}, the number of polynomials
		with $\tau_j < \beta(q) \leq 2 \tau_j$ is at most $\tau_j^{-A_d (n-1)^{d-1}}$.  For those polynomials 
		with $\beta(q) > 1/2$, we have the same bound with $\tau = 1/2$.  Therefore, the contribution 
		from these polynomials with large bias is at most
		\begin{align*}
			2^{-K} \left( \sum_j \tau_j^{-A_d (n-1)^{d-1}} (1 + \min\{2 \tau_j, \beta_d\})^m \right) + 2^{-K} 2^{A_d (n-1)^{d-1}} (1 + \beta_d)^m.	
		\end{align*}
		Since $m \leq K$ and every exponent base is at most $1 + \beta_d$, 
		\[
		2^{-K} (1 + \min\{2 \tau_j, \beta_d\})^m \leq \left( \frac{1 + \beta_d}{2} \right)^K = (1 - 2^{-d})^K.
		\]
		Also, $\tau_j \geq 1/m \geq 1/K$ so
		\[
		\tau_j^{-A_d (n-1)^{d-1}} \leq K^{A_d (n-1)^{d-1}} = \exp(O_d((n-1)^{d-1} \log n)).
		\]
		The number of dyadic classes is $O(\log K)$ and $K = \Theta_d((n-1)^d)$.  Thus, the entire 
		contribution from the polynomials with large bias is $e^{-c_d (n-1)^d}$ for some $c_d >0$.  
		Combining the contributions completes the proof.  
		
		\noindent 
		\emph{Transfer to real rank.}
		It remains to prove the rank statement over $\R$.  Regard $z_i$ as an integer
		in $\{0,1\}$.  Set 
		\[
		y_i = 1 - 2 z_i.
		\]
		Then, with the sum taken over all subsets including $T=S$,
		\[
		\chi_S(y) = \prod_{i \in S} ( 1- 2z_i) = \sum_{T \subset S} (-2)^{|T|} z^T.
		\]
		This is an invertible real change of basis: in an inclusion-compatible ordering it is trinagular with diagonal entries $(-2)^{|S|}$.  If the monomial matrix has full column rank over $\mathbb{F}_2$, some $m \times m$ minor has odd integer determinant.  
		That determinant is also nonzero over $\R$ so the monomial matrix has full real rank so we get full row rank for $M_m$.  
	\end{proof}
	
	For convenience, we let
	\[
	\Delta_{n,d} = C_d K(n-1, d)2^{-(n-1)/2} + e^{-c_d (n-1)^d}.
	\]
	Thus, for every integer $1 \leq s < K$,
	\begin{equation} \label{eq:rankbound}
		\P(\rank M_{K-s} < K-s) \leq 2^{-s} + \Delta_{n,d}.  
	\end{equation}
	We will apply this estimate with $s$ of order $\log n$.  
	
	\section{Fixed low-dimensional space is killed by many samples} \label{sec:lowdim}
	We work over $\R$ on $\{\pm 1\}^{n-1}$.  
	\begin{lemma} \label{lem:supportR}
		If $f \in \mathcal{P}_{n-1,d}$ is non-zero then for $\bm{y}$ uniformly 
		distributed on $\{\pm 1\}^{n-1}$
		\[
		\P(f(Y) \neq 0) \geq 2^{-d}.
		\]
	\end{lemma}
	\begin{proof}
		We omit this proof as it is identical to the proof of Lemma \ref{lem:support}.
	\end{proof}
	The next lemma is technical and will be used in the proof of Lemma \ref{lem:injective}.
	\begin{lemma} \label{lem:domination}
		Let $\bm{X}_1, \dots, \bm{X}_m \in \{0,1\}$ be adapted to a filtration $(\mathcal{F}_t)_{t = 0}^m$.  Let $L \geq 1$ be an integer and $ p \in [0,1]$.  Suppose that, whenever $\sum_{j \leq t} \bm{X}_j < L$, 
		\[
		\P(X_{t+1} = 1| \mathcal{F}_t) \geq p,
		\]
		then
		\[
		\P\left(\sum_{j=1}^m X_j < L \right) \leq \P(\text{Bin}(m,p) < L).  
		\]
	\end{lemma}
	
	\begin{proof}
		This is a standard stochastic domination argument.  The case $p=0$ is immediate so assume $p > 0$.  Set $\bm{S}_t =  \sum_{j=1}^t \bm{X}_j$ and $\bm{a}_t = \P(\bm{X}_t = 1| \mathcal{F}_{t-1})$.  We enlarge the probability space by uniform random variables $\bm{u}_1, \ldots, \bm{u}_m \sim U([0,1])$ that are independent of each other and the original probability space.  Define
		\[
		\bm{B}_t = \begin{cases}
			\bm{X}_t \indicator{\bm{u}_t \leq p/\bm{a}_t} & \bm{S}_{t-1} < L \\
			\indicator{\bm{u}_t \leq p} & \bm{S}_{t-1} \geq L.
		\end{cases}
		\] 
		The first case is well-defined because $\bm{a}_t \geq p$.  For $\mathcal{G}_t = \mathcal{F}_t \vee \sigma(\bm{u}_1, \ldots, \bm{u}_m)$, we have $\P(\bm{B}_t = 1| \mathcal{G}_{t-1}) = p$.  As a consequence, the $\bm{B}_t$ are independent BErnoulli random variables of parameter $p$.  On the event $\{\bm{S}_m < L\}$, we have
		$\bm{S}_{t-1} < L$ for every $t$, so $\bm{B}_t \leq \bm{X}_t$ throughout.  Therefore,
		\[
		\{\bm{S}_m  < L\} \subset \left\{\sum_{t=1}^m \bm{B}_t < L\right\},
		\]  
		which concludes the proof.
	\end{proof}
	The following lemma indicates that with enough random samples, the evaluation map is injective.  In other words,
	the only polynomial in the subspace that can vanish on all the sample points is the zero polynomial.
	
	\begin{lemma} \label{lem:injective}
		Fix $d \geq 1$ and set $p_d = 2^{-d}$.  Let $\mathcal{U} \subset \mathcal{P}_{n-1,d}$ be a fixed subspace of dimension $L$.  If $\bm{y}^{(1)}, \dots, \bm{y}^{(m)}$ are uniform, independent samples from $\{\pm 1\}^{n-1}$ and 
		\[
		L \leq \frac{p_d m}{2},
		\]
		then
		\[
		\P(\mathsf{Ev}_m|_\mathcal{U} \text{ is not injective}) \leq \exp(-p_d m/8).  
		\]
	\end{lemma}
	
	\begin{proof}
		For $L=0$ the conclusion is immediate, so assume that $\dim L > 0$.  Let $\bm{R}_t$ be the rank  of the first $t$ evaluation functionals restricted to $\mathcal{U}$.  To be more explicit, once the sample points $\bm{y}^{(1)}, \ldots, \bm{y}^{(t)}$ are fixed, each defines a linear functional, $\ell_i$, from $\mathcal{U}$ to $\R$ by $\ell_i(f) = f(\bm{y}^{(i)})$.  Thus, 
		\[
		\bm{R}_t = \dim (\text{span}\{\ell_1, \ldots \ell_t\}).
		\]
		If $\bm{R}_t < L$, choose a non-zero $f_t$ in their common kernel.  This choice can be made measurably with respect to $\mathcal{F}_t = \sigma(\bm{y}^{(1)}, \ldots, \bm{y}^{(t)})$, for example by taking the first vector in a row-reduced basis of the kernel.  Conditioning upon the first $t$ samples, $f_t$ is fixed and $\bm{y}^{(t+1)}$ is still uniform.  By Lemma \ref{lem:supportR},
		\[
		\P(f_t(\bm{y}^{(t+1)})  \neq 0| \mathcal{F}_t) \geq p_d.  
		\]
		On this event, the new evaluation functional is not in the span of the preceding ones so $\bm{R}_{t+1} = \bm{R}_t + 1$.  Lemma \ref{lem:domination} gives
		\[
		\P(\bm{R}_m < L) \leq \P\big(\text{Bin}(m,p_d)< L\big) \leq \P\big(\text{Bin}(m,p_d) < p_d m/2\big).
		\]
		Chernoff's bound can be used to control this last probability by $e^{-p_d m/8}$.  
	\end{proof}

	\section{Structure of the kernel} \label{sec:kernel}
	In this section, we study the subspace of low-degree polynomials that
	vanish at all the $K(n-1,d) - s$ sample points seen so far.  We show that every polynomial
	in this subspace has many degree $d$ monomials involving mutually disjoint variables.

	Let $f \in \mathcal{P}_{n-1, d}$ and write its degree-$d$ part as
	\[
	f^{=d}(y) = \sum_{|S| = d} \hat{f}(S) \chi_S(y).
	\]
	We introduce the associated $d$-uniform hypergraph as
	\[
	\mathcal{H}_d(f) := \{S \subset [n-1]: |S| = d, \hat{f}(S) \neq 0\}
	\]
	be the hypergraph representing the degree-$d$ part of $f$.  
	Let $\nu_d(f)$ be the matching number of $\mathcal{H}_d(f)$, meaning the maximum number of 
	pairwise disjoint members of $\mathcal{H}_d(f)$.	This notion is closely related to the \emph{rank} of a polynomial defined in \cite{MR3542863} which will play a role in the next section.  If $\deg f < d$, then $\nu_d(f) = 0$.  
	
	\begin{proposition} \label{prop:matching}
		Fix $d \geq 1$ and set $K = K(n-1, d)$.  Let $m_0$ be any integer such that
		$K/2 \leq m_0 \leq K-1$.  For uniformly distributed, independent $\bm{y}^{(1)}, \dots, \bm{y}^{(m_0)} \in \{\pm 1\}^{n-1}$. 
		We denote
		\[
		\mathcal{W}_0 = \ker \mathsf{Ev}_{m_0}
		\]
		Then there exist
		constants $\kappa_d, c'_d > 0$ such that
		\[
		\P(\exists \text{ non-zero } f \in \mathcal{W}_0: \nu_d(f) < \kappa_d (n-1)) \leq e^{-c'_d (n-1)^d}
		\]
		for all sufficiently large $n$.  
	\end{proposition}
	
	\begin{proof}
		Let $p_d = 2^{-d}$.  Let $\alpha_d>0$ be a sufficiently small constant whose value we will fix below.  Define
		\[
		R = \left\lfloor \frac{\alpha_d (n-1) }{d} \right\rfloor.
		\] 
		For $J \subset [n-1]$, we define the deterministic subspace
		\[
		\mathcal{U}_J := \{f \in \mathcal{P}_{r,d}: \hat{f}(S) = 0 \text{ whenever } |S| =d \text{ and } 
		S \cap J = \emptyset\}.
		\]
		Thus, every set in the Fourier support of $f^{=d}$, for $f \in \mathcal{U}_J$, meets $J$.  
		The dimension of $\mathcal{U}_J$ is  
		\begin{align*}
			L_J &= \dim \mathcal{U}_J = K(n-1,d-1) + \binom{n-1}{d} - \binom{n-1-|J|}{d} \\
			&\leq K(n-1,d-1) + |J| \binom{n-1-1}{d-1}.
		\end{align*}
		The inequality follows by assigning to every $d$-set meeting $J$ one distinguished vertex in its intersection with $J$. 
		
		For fixed $d$,
		\[
		\frac{K(n-1,d-1)}{K(n-1,d)} \rightarrow 0 \text{ and } \frac{\binom{n-1-1}{d-1}}{K(n-1,d)} \leq 
		\frac{\binom{n-2}{d-1}}{\binom{n-1}{d}} = \frac{d}{n-1}.
		\]
		Choose $\alpha_d$ small enough such that $d \alpha_d \leq p_d/16$.  For all sufficiently large $n$, 
		we also have $K(n-1,d-1) \leq p_d K/16$ and $m_0 \geq K/2$.  Therefore, whenever $|J| \leq \alpha_d (n-1)$,
		\[
		L_J \leq \frac{p_dK}{8}\leq \frac{p_d m_0}{2}.
		\]
		Lemma \ref{lem:injective} allows us to deduce that for every fixed $J$,
		\[
		\P(\mathcal{W}_0 \cap \mathcal{U}_J \neq \{0\} )\leq e^{-p_d m_0/8}.
		\]
		
		Now, we take a union bound over $J$.  If $d \geq 2$, then
		\[
		\sum_{j \leq \alpha_d (n-1)} \binom{n-1}{j} \leq 2^{n-1},
		\]
		and $m_0 = \Theta_d((n-1)^d)$.  Thus, 
		\begin{align} \label{eq:intersectUJ}
			\P(\exists J, |J| \leq \alpha_d (n-1): \mathcal{W}_0 \cap \mathcal{U}_J \neq \{0\})
			&\leq 2^{n-1} e^{-p_d m_0/8} \nonumber \\
			&\leq e^{-c'_d (n-1)^d}.
		\end{align}
		For $d = 1$, one can choose $\alpha_1$ even smaller so that the binary entropy
		$h(\alpha_1) = -\alpha_1 \log \alpha_1 - (1 - \alpha_1) \log (1 - \alpha_1)$ satisfies
		$h(\alpha_1) < 1/64$.  Then,
		\[
		\sum_{j \leq \alpha_1 (n-1)} \binom{n-1}{j} \leq e^{h(\alpha_1) (n-1)}
		\]  
		and since $p_1 = 1/2$ and $m_0 = \Theta(n)$, the same union bound is at most $e^{-c'_1 (n-1)}$.
		
		We now work on the complement of the event in \eqref{eq:intersectUJ}.  Suppose that $0 \neq f \in \mathcal{W}_0$
		and $\nu_d(f) < R$.  Take a maximal matching in $\mathcal{H}_d(f)$ and let $J$ be the 
		union of its edges.  Then,
		\[
		|J| < dR \leq \alpha_d (n-1).  
		\]
		Maximality implies that every edge of $\mathcal{H}_d(f)$ intersects $J$.  Otherwise an edge disjoint from $J$ could be added to the matching.  Thus, $f \in \mathcal{U}_J$ contradicting 
		$\mathcal{W}_0 \cap \mathcal{U}_J = \{0\}$.  We conclude that every non-zero $f \in \mathcal{W}_0$
		satisfies
		\[
		\nu_d(f) \geq R \geq \kappa_d (n-1)
		\]
		for some constant $\kappa_d > 0$ and all large $n$.  
		
		Note that this also excludes non-zero polynomials of degree less than $d$ from the kernel as such a polynomial belongs to $\mathcal{U}_\emptyset$.  
	\end{proof}
	
	\section{Polynomial Anti-concentration and the last exposure} \label{sec:lastexposure}
	Anti-concentration measures how unlikely a random variable is to land in a particular set.  Littlewood-Offord theory studied the anti-concentration of sums of independent random variables \cite{MR9656, MR14608}.  Anti-concentration of linear and quadratic sums has played an important role in random matrix theory \cite{MR2480613, MR2407948, MR2955045, MR4273471, MR4076632}. 
	We invoke the theory of anti-concentration for higher-degree polynomials.  
	\begin{definition}\label{def:rank}
		For a multi-linear polynomial $q$ of the form
		\begin{equation} \label{eq:polydef}
			q(x_1, \dots, x_k) = \sum_{S \subset [k]: |S| \leq d} a_S x^S,
		\end{equation}
		the \emph{rank} of $P$ is the largest integer $r$ such that there exist disjoint sets
		$S_1, \dots, S_r \subset [k]$ of size $d$ with $|a_{S_j}| > 1$ for $j \in [r]$.  
	\end{definition}
	\begin{theorem}\cite[Theorem 1.6]{MR3542863} \label{thm:polyanticoncentration}
		There is an absolute constant $B$ such that the following holds for all $d$ and $n$.  Let $q$ 
		be a polynomial as in \eqref{eq:polydef}.
		If the rank, $r$, of $P$ is greater than or equal to 2, then for any interval $I$ of length 1, and for $\bm{x}$ uniformly distributed in $\{\pm 1\}^k$, 
		\begin{equation} \label{eq:polyanticoncentration}
			\P(q(\bm{x}) \in I) \leq \min\left\{\frac{B d^{4/3} \sqrt{\log r}}{r^{1/4d+1}}, \frac{\exp(B d \log d \log \log r+ B d^2 \log d)}{\sqrt{r}}\right\}.
		\end{equation}
	\end{theorem}
	
	In our setting of fixed $d$, the second term on the right-hand side of \eqref{eq:polyanticoncentration} will be the tighter bound.  We record this consequence 
	in the following proposition.  
	
	\begin{proposition}\label{prop:anticoncentration}
		There is an absolute constant $B$ such that the following holds.  Let $f \in \mathcal{P}_{n-1,d}$
		and suppose that $\mathcal{H}_d(f)$ contains $R = \nu_d(f) \geq 2$ pairwise disjoint edges with non-zero 
		coefficients.  Then for $\bm{y}$ uniformly distributed in $\{\pm 1\}^n-1$,
		\begin{equation} \label{eq:secondbranch}
			\sup_{u \in \R} \P(f(\bm{y}) = u) \leq \frac{\exp(B d \log d \log \log R + B d^2 \log d)}{\sqrt{R}}.
		\end{equation}
		Consequently, for the constant $\kappa_d$ in Proposition \ref{prop:matching}, there exists 
		a constant $C_d > 0$ such that whenever $\nu_d(f) \geq \kappa_d (n-1)$ and $n$ is sufficiently
		large,
		\begin{equation} \label{eq:qanticoncentration}
			\sup_{u \in \R} \P(f(\bm{y} = u)) \leq q_{n,d}
		\end{equation}
		where 
		\[
		q_{n,d} = C_d \frac{(\log (n-1))^{a_d}}{\sqrt{n-1}} \quad (a_d = B d \log d).  
		\]
	\end{proposition}
	
	\begin{proof}
		Fix $u \in \R$ and take a maximal matching in $S_1, \ldots, S_R$ in $\supp (\widehat{f^{=d}})$.  Let 
		\[
		\delta = \min_{1 \leq j \leq R} |\hat{f}(S_j)| > 0
		\]
		and now define
		\[
		g(y) = \frac{f(y) - u}{\delta}
		\]
		The selected degree-$d$ coefficients of $g$ have magnitude at least 1 so the polynomial rank in the sense of definition \ref{def:rank} is at least $R$.  We apply Theorem \ref{thm:polyanticoncentration} to the open interval $(-1/2, 1/2)$ with the second term in the minimum and note that
		\[
		\{f(\bm{y})  = u\} = \{g(\bm{y}) = 0\} \subset \{g(\bm{y}) \in (-1/2, 1/2)\}.
		\]
		This proves \eqref{eq:secondbranch}.
		
		Since $\kappa_d (n-1) \leq R \leq (n-1)/d$,
		\[
		\frac{\exp(B d \log d \log \log R + B d^2 \log d)}{\sqrt{R}} \leq \frac{e^{B d^2 \log d}}{\sqrt{\kappa_d}} \frac{ (\log (n-1))^{B d \log d}}{\sqrt{n-1}}.
		\]
		Thus, any $C_d \geq e^{Bd^2 \log d}/\sqrt{\kappa_d}$ suffices for \eqref{eq:qanticoncentration}.  
	\end{proof}

	\subsection{Sequential exposures and a union bound}
	In this section we show that each new row of $M$, corresponding to an independent sample, is likely to increase the rank.  Once the total number of rows is $K$, we can conclude that $M_K$ is non-singular.  
	
	\begin{proposition} \label{prop:errorbound}
		Fix $d \geq 1$ and let $n$ be sufficiently large and set $K = K(n-1, d)$.  
		For every integer $1 \leq s \leq \lfloor K/2 \rfloor$,
		\[
		\P(M_K \text{ is singular}) \leq 2^{-s} + C_d K 2^{-(n-1)/2} + e^{-c_d (n-1)^d} + e^{-c'd (n-1)^{d}} + s q_{n-1, d}.
		\]
		The first three terms originate from Proposition \ref{prop:firstmoment}, the fourth from Proposition \ref{prop:matching}, and the last from the union bound over the final $s$ samples.  
	\end{proposition}  
	\begin{proof}
		Set $m_0 = K - s$.  We expose the first $m_0$ samples.  By \eqref{eq:rankbound}, outside an 
		event of probability at most $2^{-s} + \Delta_{n, d}$, we have
		\[
		\rank M_{m_0} = m_0 \quad \text{and} \quad  \dim \mathcal{W}_0 = s
		\]
		where we recall that $\mathcal{W}_0 = \ker \mathsf{Ev}_{m_0}$.  
		Since $K/2 \leq m_0 \leq K-1$, we can apply Proposition \ref{prop:matching} to conclude that
		outside of an event with probability bounded above $e^{-c'_d (n-1)^d}$, every non-zero $f \in \mathcal{W}_0$ has $\nu_d(f) \geq \kappa_d (n-1)$.  Let $\mathcal{E}$ be the intersection of these
		two good events.  $\mathcal{E}$ depends only on the first $m_0$ samples and 
		\begin{equation} \label{eq:badevent}
			\P(\mathcal{E}^c) \leq 2^{-s} + \Delta_{n,d} + e^{-c'_d (n-1)^d}.
		\end{equation}
		On $\mathcal{E}$, we expose the remaining $s$ samples one at a time and define for $1 \leq j \leq s$,
		\[
		\mathcal{W}_j = \{f: \mathcal{W}_{j-1}: f(\bm{y}^{(m_0 + j)}) = 0\}.
		\]
		A new evaluation is a single linear functional so the kernel dimension can drop by zero or one.  
		In particular,
		\[
		\dim \mathcal{W}_{j-1} \geq s - j + 1 > 0
		\]
		for $1 \leq j \leq s$.  Choose a non-zero $f_j \in \mathcal{W}_{j-1}$ by some fixed, measurable rule. Since $\mathcal{W}_{j-1} \subset \mathcal{W}_0$, Proposition \ref{prop:matching} gives, on $\mathcal{E}$,
		\[
		\nu_d(f_j) \geq c_d (n-1).  
		\]
		
		Let $\mathcal{F}_t = \sigma(\bm{y}^{(1)}, \ldots, \bm{y}^{(t)})$.  Conditioning on $\mathcal{F}_{m_0 + j - 1}$, the function $f_j$ is fixed and the next sample is independent and uniform.  If the new evaluation functional vanishes identically on $\mathcal{W}_{j-1}$ then in particular, $f_j(\bm{y}^{m_0 + j}) = 0$.  Proposition \ref{prop:anticoncentration} gives, on 
		$\mathcal{E}$ that 
		\[
		\P(\dim \mathcal{W}_j = \dim \mathcal{W}_{j-1} | \mathcal{F}_{m_0 + j - 1)} \leq q_{n-1,d}).
		\]
		If none of these $s$ possible failures occur, the dimension drops from $s$ to zero.  A conditional union bound therefore yields
		\begin{equation} \label{eq:conditionalunion}
			\P(\mathcal{E} \cap \{\mathcal{W}_s \neq \{0\}) \leq s q_{n-1, d}.
		\end{equation}
		Finally, $\mathcal{W}_s = \ker \mathsf{Ev}_K$ and a linear map between two $K$-dimenisonal spaces is singular exactly when its kernel is non-zero.  Combining \eqref{eq:badevent} and \eqref{eq:conditionalunion} completes the proof.
	\end{proof}

	\subsection{Choosing the parameters and completing the proof of Theorem \ref{thm:characterevaluation}}
	\begin{proof}[Proof of Theorem \ref{thm:characterevaluation}] 
		Choose $s = \left \lceil \frac{1}{2} \log_2 (n-1) \right \rceil$.  Since $K \geq n$, this choice
		satisfies
		\[
		1 \leq s \leq \lfloor K/2 \rfloor
		\]
		for all sufficiently large $n$.  Moreover,
		\[
		2^{-s} \leq (n-1)^{-1/2} \quad \text{ and } \quad s \leq 1 + \frac{\log (n-1)}{2 \log 2}.
		\]
		Proposition \ref{prop:errorbound} therefore gives
		\begin{align*}
			\P(M_K \text{ is singular}) &\leq \frac{1 + C_d (1 + \frac{\log (n-1)}{2 \log 2})(\log (n-1))^{a_d}}{\sqrt{r}} \\
			&\quad \quad + C_d K 2^{(n-1)/2} + e^{c_d (n-1)^d} + e^{-c'_d (n-1)^d}.
		\end{align*}
		As $K = O_d(n^d)$, 
		\[
		C_d K 2^{(n-1)/2} \leq C_d' e^{-(\log 2)(n-1)/4}.
		\]
		Also, $(n-1)^d \geq (n-1)$ so if we let
		\[
		b_d = \min\{\log 2/ 4, c_d, c'_d\} > 0 
		\]
		then after adjusting the constant $C_d$, we have
		\[
		\P(M_K \text{ is singular}) \leq C_d \frac{(\log n)^{a_d + 1}}{\sqrt{n}} + C_d e^{-b_d n}.
		\]
	\end{proof}
	
	\section{Applications} \label{sec:applications}
	\subsection{Simplicial faces of correlation matrices}\label{sec:elliptope}
	Set $d=2$ and $K=D(n,2)=1+\binom n2$.  For sign vectors $s^{(1)},\ldots,s^{(n)}\in\{\pm 1\}^m$, define
	\begin{equation}\label{eq:schur}
		S=[s^{(1)}\ \cdots\ s^{(n)}],\qquad
		Z=\bigl[\bm{1}_m\ \ (s^{(i)}\odot s^{(j)})_{i<j}\bigr]\in\R^{m\times K},
	\end{equation}
	where $\odot$ is coordinatewise multiplication. The family is \emph{Schur independent} if $\rank Z=K$ \cite[Definition~2.1]{kueng2021binary}. After transposition, $n$ counts components and $m$ is their ambient dimension.
	
	\begin{lemma}\label{lem:schur}
		For independent uniform $\bm{s}^{(1)},\ldots,\bm{s}^{(n)}\in\{\pm 1\}^m$, and all sufficiently large $n$, uniformly in $m\ge D(n,2)$,	$\P[\rank\bm{Z}<K]\le C_2\frac{(\log n)^{C}}{\sqrt n}$ for some universal constant $C > 0$.
	\end{lemma}
	\begin{proof}
		The rows $\x^{(a)}=(\bm{s}^{(1)}_a,\ldots,\bm{s}^{(n)}_a)$ of $\bm{S}$ are independent uniform points in $\{\pm 1\}^n$. Row $a$ of $\bm{Z}$ is $(1,(\x^{(a)}_i\x^{(a)}_j)_{i<j})$, the distinct coordinates of
		$(\x^{(a)})^{\otimes2}$. Deleting repeated tensor coordinates preserves rank. At $m=K$, invertibility of $\bm{Z}$ is therefore exactly independence of $K$ tensor samples so we can apply Theorem \ref{thm:main}. For $m>K$, apply the theorem to the first $K$ rows.  Additional rows cannot lower column rank. 
	\end{proof}

	\subsection{The deterministic face criterion}
	The \emph{elliptope} and the \emph{cut polytope} are, respectively,
	\[
	\mathcal{E}_m=\{H=H^\top\succeq0:\diag H=\bm{1}_m\},\qquad
	\mathcal{C}_m=\text{conv}\{ss^\top:s\in\{\pm 1\}^m\}\subseteq\mathcal{E}_m.
	\]
	An exposed face is the set minimizing a linear functional over a convex set; it is \emph{simplicial} when its vertices are affinely independent. The next criterion is due to Laurent--Poljak \cite{laurent1996facial}; see Tropp \cite[Fact~3.1]{tropp2018simplicial}. We include the argument to specify what the rank hypothesis provides.
	
	\begin{lemma}\label{lem:face}
		If the family in~\eqref{eq:schur} is Schur independent, then $\mathcal{F}=\text{conv}\{s^{(i)}(s^{(i)})^\top:1\le i\le n\}$ is an exposed simplicial face of $\mathcal{E}_m$ of dimension $n-1$.
	\end{lemma}
	\begin{proof}
		First, since $S$ has full column rank, multiply a relation $\sum_i b_i s^{(i)}=0$ coordinatewise by $s^{(1)}$ to obtain a relation among $\bm{1}_m,s^{(1)}\odot s^{(2)},\ldots,s^{(1)}\odot s^{(n)}$, which are columns of $Z$. Thus every $b_i=0$.
		
		Let $Q$ be the orthogonal projector onto $\text{range}(S)^\perp$. For $H\succeq0$, $\tr(QH)\ge0$, with equality precisely when
		$\text{range}(H)\subseteq\text{range}(S)$. Consequently,
		\[
		\mathcal{G}=\{H\in\mathcal{E}_m:\tr(QH)=0\}
		\]
		is an exposed face. Each $H\in\mathcal{G}$ has a unique representation $H=SAS^\top$ with $A=A^\top\succeq0$, since $S$ has a left inverse. The diagonal equations are
		\[
		(\tr A-1)\bm{1}_m+2\sum_{i<j}A_{ij} (s^{(i)}\odot s^{(j)})=0.
		\]
		Schur independence forces $A_{ij}=0$ for $i\ne j$ and $\tr A=1$. Positive semidefiniteness gives $A_{ii}\ge0$. Thus,
		\begin{equation}\label{eq:face-description}
			\mathcal{G}=\{S\diag(\tau)S^\top:\tau_i\ge0,\ \sum_i\tau_i=1\}=\mathcal{F}.
		\end{equation}
		The map $A\mapsto SAS^\top$ is injective, so this is an affine copy of the standard $(n-1)$-simplex.
	\end{proof}
	
	\begin{theorem}[Random faces at the exact threshold]\label{thm:faces}
		For sufficiently large $n$ and every $m\ge D(n,2)$, independent uniform $\bm{s}^{(1)},\ldots,\bm{s}^{(n)}\in\{\pm 1\}^m$ satisfy
		\begin{align*}
		\P&\left[\text{conv}\{\bm{s}^{(i)}(\bm{s}^{(i)})^\top\}_{i=1}^n
		\text{ is an exposed simplicial face of }\mathcal{E}_m \text{ of dimension }n-1\right] \\
		&\qquad \qquad \qquad \qquad \qquad \qquad \qquad \qquad \qquad \qquad \qquad \qquad \geq 1-C_2\frac{(\log n)^{C}}{\sqrt n}.
		\end{align*}
		Every face obtained this way is also a face of $\mathcal{C}_m$.
	\end{theorem}
	\begin{proof}
		Combine Lemmas~\ref{lem:schur} and~\ref{lem:face}. The same exposing
		functional works on $\mathcal{C}_m$, since $\mathcal{F}\subseteq\mathcal{C}_m\subseteq\mathcal{E}_m$.
	\end{proof}

	The classical dimension bound \cite[Fact~2.7]{tropp2018simplicial} is $\ell(\ell+1)\le2(m-1)$ for a simplicial face of dimension $\ell$. For $\ell=n-1$, this is exactly $m\ge D(n,2)$.
	Thus Theorem~\ref{thm:faces} reaches the first geometrically possible ambient dimension. In particular, set
	\[
	n_*(m)=\left\lfloor\frac{1+\sqrt{8m-7}}2\right\rfloor.
	\]
	With $n_*(m)$ random sign vertices, the resulting face has the maximum possible \emph{simplicial} dimension, $n_*(m)-1$, except with probability $C(\log m)^{C}m^{-1/4}$.
	These common faces identify portions of the cut-polytope geometry preserved by its semidefinite relaxation.

	\section{Sign-component identifiability and recovery}\label{sec:recovery}
	The second application is to Kueng--Tropp's sign-component factorization \cite{kueng2021binary}. A simplicial face does more than give unique weights in a known family.  The face property prevents alternative positive decompositions from using different extreme points.
	
	\begin{corollary}[Unknown sign components at the threshold]\label{cor:components}
		Under the hypotheses of Theorem~\ref{thm:faces}, with probability at least $1-C_2 \log^C n/ \sqrt{n}$ the following holds simultaneously for every $\tau_i>0$ with $\sum_i\tau_i=1$:
		\begin{equation}\label{eq:mixture}
			A_\tau=\sum_{i=1}^n\tau_i\bm{s}^{(i)}(\bm{s}^{(i)})^\top =\bm{S}\diag(\tau)\bm{S}^\top
		\end{equation}
		has rank $n$ and a unique positive sign-component representation, up to permutation, sign changes of the vectors, and merging repeated identical rank-one terms. The competing representation need not have $n$ components.
		For each fixed such input, Kueng--Tropp's polynomial-time Algorithm~1 recovers the components and weights with probability one over its internal randomness.
	\end{corollary}
	\begin{proof}
		We condition on Schur independence and drop the boldface. The matrix $S$ has rank $n$, and all weights are positive, so $\rank A_\tau=n$. If $A_\tau=\sum_j\alpha_j t^{(j)}(t^{(j)})^\top$ is another positive
		sign decomposition, the diagonal enforces $\sum_j\alpha_j=1$. Since $A_\tau\in\mathcal{F}$ and $\mathcal{F}$ is a face of $\mathcal{E}_m$, every alternative component matrix lies in $\mathcal{F}$. A rank-one sign matrix is an extreme point of $\mathcal{E}_m$ and must therefore be a vertex of this simplex. After identical terms are merged, uniqueness of barycentric coordinates fixes their weights. Rank-one sign matrices determine their sign vectors
		up to a global sign. The algorithmic conclusion is exactly \cite[Theorem~I and Theorem~3.7]{kueng2021binary}, since Schur independence is its hypothesis.
	\end{proof}

	For positive weights, $\text{range}(A_\tau)=\text{range}(S)$. The projector $Q$ onto $\ker A_\tau$ is therefore determined by the data, and \eqref{eq:face-description} gives
	\[
	\mathcal{F}=\{H\in\mathcal{E}_m:\tr(QH)=0\}.
	\]
	Kueng--Tropp optimize a random linear functional over this face to find one vertex, deflate that component, and repeat; the recovered vertices then determine the weights~\cite[Sections~3.6--3.7]{kueng2021binary}.

	We briefly compare our corollary with the best previous bound.  
	In the notation here, Tropp's Theorems~2.9 and~2.11 for uniform signs bound failure of the simplicial-face conclusion by, respectively,
	\begin{equation}\label{eq:tropp-bounds}
		n^2e^{-m/n^2},\qquad 4\exp\left(\frac{n^2-c_0m/16}{4}\right),
	\end{equation}
	where $c_0>0$ is an absolute constant \cite{kueng2021binary}. Thus, they require $m \geq C_0 n^2$ samples, for which our result brings the unspecified constant down to the exact threshold.

	\bibliographystyle{plain}
	\bibliography{tensor.bib}

\begin{thebibliography}{10}

\bibitem{MR3400278}
Emmanuel Abbe, Amir Shpilka, and Avi Wigderson.
\newblock Reed-{M}uller codes for random erasures and errors.
\newblock {\em IEEE Trans. Inform. Theory}, 61(10):5229--5252, 2015.

\bibitem{anari2018smoothed}
Nima Anari, Constantinos Daskalakis, Wolfgang Maass, Christos Papadimitriou,
  Amin Saberi, and Santosh Vempala.
\newblock Smoothed analysis of discrete tensor decomposition and assemblies of
  neurons.
\newblock {\em Advances in neural information processing systems}, 31, 2018.

\bibitem{anderson2014more}
Joseph Anderson, Mikhail Belkin, Navin Goyal, Luis Rademacher, and James Voss.
\newblock The more, the merrier: the blessing of dimensionality for learning
  large gaussian mixtures.
\newblock In {\em Conference on Learning Theory}, pages 1135--1164. PMLR, 2014.

\bibitem{BV2018v1}
P.~Baldi and R.~Vershynin.
\newblock Boolean polynomial threshold functions and random tensors, 2018.
\newblock arXiv:1803.10868v1.

\bibitem{MR4016132}
Pierre Baldi and Roman Vershynin.
\newblock Polynomial threshold functions, hyperplane arrangements, and random
  tensors.
\newblock {\em SIAM J. Math. Data Sci.}, 1(4):699--729, 2019.

\bibitem{MR4438378}
Aditya Bhaskara, Aidao Chen, Aidan Perreault, and Aravindan Vijayaraghavan.
\newblock Smoothed analysis for tensor methods in unsupervised learning.
\newblock {\em Math. Program.}, 193(2):549--599, 2022.

\bibitem{MR2557947}
Jean Bourgain, Van~H. Vu, and Philip~Matchett Wood.
\newblock On the singularity probability of discrete random matrices.
\newblock {\em J. Funct. Anal.}, 258(2):559--603, 2010.

\bibitem{MR14608}
P.~Erd\"os.
\newblock On a lemma of {L}ittlewood and {O}fford.
\newblock {\em Bull. Amer. Math. Soc.}, 51:898--902, 1945.

\bibitem{MR4273471}
Asaf Ferber, Vishesh Jain, Kyle Luh, and Wojciech Samotij.
\newblock On the counting problem in inverse {L}ittlewood-{O}fford theory.
\newblock {\em J. Lond. Math. Soc. (2)}, 103(4):1333--1362, 2021.

\bibitem{MR4356701}
Vishesh Jain, Ashwin Sah, and Mehtaab Sawhney.
\newblock Singularity of discrete random matrices.
\newblock {\em Geom. Funct. Anal.}, 31(5):1160--1218, 2021.

\bibitem{MR1260107}
Jeff Kahn, J\'anos Koml\'os, and Endre Szemer\'edi.
\newblock On the probability that a random {$\pm 1$}-matrix is singular.
\newblock {\em J. Amer. Math. Soc.}, 8(1):223--240, 1995.

\bibitem{MR238371}
J.~Koml\'os.
\newblock On the determinant of random matrices.
\newblock {\em Studia Sci. Math. Hungar.}, 3:387--399, 1968.

\bibitem{kueng2021binary}
Richard Kueng and Joel~A Tropp.
\newblock Binary component decomposition part i: the positive-semidefinite
  case.
\newblock {\em SIAM Journal on Mathematics of Data Science}, 3(2):544--572,
  2021.

\bibitem{MR4068053}
Stefan Kunis, H.~Michael M\"oller, and Ulrich von~der Ohe.
\newblock Prony's method on the sphere.
\newblock {\em SMAI J. Comput. Math.}, S5:87--97, 2019.

\bibitem{kwan2025resolution}
Matthew Kwan and Lisa Sauermann.
\newblock Resolution of the quadratic littlewood--offord problem.
\newblock {\em Compositio Mathematica}, 161(12):3089--3139, 2025.

\bibitem{laurent1996facial}
Monique Laurent and Svatopluk Poljak.
\newblock On the facial structure of the set of correlation matrices.
\newblock {\em SIAM Journal on Matrix Analysis and Applications},
  17(3):530--547, 1996.

\bibitem{MR9656}
J.~E. Littlewood and A.~C. Offord.
\newblock On the number of real roots of a random algebraic equation. {III}.
\newblock {\em Rec. Math. [Mat. Sbornik] N.S.}, 12/54:277--286, 1943.

\bibitem{MR3542863}
Raghu Meka, Oanh Nguyen, and Van Vu.
\newblock Anti-concentration for polynomials of independent random variables.
\newblock {\em Theory Comput.}, 12:Paper No. 11, 16, 2016.

\bibitem{MR2955045}
Hoi~H. Nguyen.
\newblock On the least singular value of random symmetric matrices.
\newblock {\em Electron. J. Probab.}, 17:no. 53, 19, 2012.

\bibitem{o2014analysis}
Ryan O'Donnell.
\newblock {\em Analysis of boolean functions}, volume~2.
\newblock Cambridge University Press Cambridge, 2014.

\bibitem{MR2407948}
Mark Rudelson and Roman Vershynin.
\newblock The {L}ittlewood-{O}fford problem and invertibility of random
  matrices.
\newblock {\em Adv. Math.}, 218(2):600--633, 2008.

\bibitem{MR2480613}
Terence Tao and Van~H. Vu.
\newblock Inverse {L}ittlewood-{O}fford theorems and the condition number of
  random discrete matrices.
\newblock {\em Ann. of Math. (2)}, 169(2):595--632, 2009.

\bibitem{MR4076632}
Konstantin Tikhomirov.
\newblock Singularity of random {B}ernoulli matrices.
\newblock {\em Ann. of Math. (2)}, 191(2):593--634, 2020.

\bibitem{tropp2018simplicial}
Joel~A Tropp.
\newblock Simplicial faces of the set of correlation matrices.
\newblock {\em Discrete \& Computational Geometry}, 60(2):512--529, 2018.

\bibitem{MR4140540}
Roman Vershynin.
\newblock Concentration inequalities for random tensors.
\newblock {\em Bernoulli}, 26(4):3139--3162, 2020.

\end{thebibliography}
	
\end{document}